\documentclass[a4paper,12pt]{amsart}

\usepackage{amsmath,amssymb,amsfonts,graphicx}
\usepackage{amsmath,amssymb,amsfonts,graphicx,color,fancyhdr}
\usepackage[latin1]{inputenc}
\usepackage{psfrag}

\newtheorem{theorem}{Theorem}[section]
\newtheorem{proposition}[theorem]{Proposition}
\newtheorem{lemma}[theorem]{Lemma}
\newtheorem{definition}[theorem]{Definition}

\newtheorem{remark}[theorem]{Remark}
\newtheorem{corollary}[theorem]{Corollary}

\usepackage[colorlinks=true, pdfstartview=FitV, linkcolor=blue, citecolor=blue, urlcolor=blue]{hyperref}
\usepackage{psfrag,caption}

\allowdisplaybreaks

\def\be#1 {\begin{equation} \label{#1}}
\newcommand{\ee}{\end{equation}}

\def\sqw{\hbox{\rlap{\leavevmode\raise.3ex\hbox{$\sqcap$}}$%
\sqcup$}}
\def\findem{\ifmmode\sqw\else{\ifhmode\unskip\fi\nobreak\hfil
\penalty50\hskip1em\null\nobreak\hfil\sqw
\parfillskip=0pt\finalhyphendemerits=0\endgraf}\fi}

\newcommand{\R}{{\mathbb {R}}}

\newcommand{\N}{{\mathbb N}}
\newcommand{\Z}{{\mathbb Z}}

\newcommand{\po}{{\mathbb P}}
\newcommand{\qo}{{\mathbb Q}}
\newcommand{\dist}{\operatorname{dist}}
\newcommand{\supp}{\operatorname{supp}}
\newcommand{\ds}{\widetilde{\operatorname{size}}_{\vec{T}}}
\newcommand{\vs}{\overrightarrow{\operatorname{size}}_{\vec{T}}}
\DeclareMathOperator*{\essup}{ess\,sup}

\begin{document}

\title
[Quantitative weighted estimates for Rubio de Francia's square function]
{\bf  Quantitative weighted estimates for Rubio de Francia's Littlewood--Paley square function}

\author[R. Garg, L. Roncal and S. Shrivastava]{Rahul Garg \and Luz Roncal \and Saurabh Shrivastava}

\address[R. Garg and S. Shrivastava]{
Dept. of Mathematics \\
Indian Institute of Science Education and Research Bhopal\\
462066, India}
\email{\{rahulgarg,saurabhk\}@iiserb.ac.in}

\address[L. Roncal]{
BCAM - Basque Center for Applied Mathematics \\
48009 Bilbao, Spain and Ikerbasque, Basque Foundation for Science, 48011 Bilbao, Spain}
\email{lroncal@bcamath.org}






\begin{abstract} We consider the Rubio de Francia's Littlewood--Paley square function associated with an arbitrary family of intervals in $\R$ with finite overlapping.
Quantitative weighted estimates are obtained for this operator. The linear dependence on the characteristic of the weight $[w]_{A_{p/2}}$ turns out to be sharp for $3\le p<\infty$, whereas the sharpness in the range $2<p<3$ remains as an open question. Weighted weak-type estimates in the endpoint $p=2$ are also provided. The results arise as a consequence of a sparse domination shown for these operators, obtained by  suitably adapting the ideas coming from \cite{Be} and \cite{CDpO}. 
\end{abstract}

\maketitle


\section{Introduction}\label{sec:intro}
Let $\Omega=\{\omega_k\}_{k\in \Z}$ be an arbitrary family of intervals in $\R$ with finite overlapping, i.e., 
\begin{equation}
\label{overl}
\sum_{k\in \Z} {\bf 1}_{\omega_k}(x)\le B, \quad \text{ for any } x\in \R,
\end{equation}
and some constant $B>0$. Here ${\bf 1}_{E}$ denotes the characteristic function of a measurable set $E\subseteq \R.$

The \textit{Littlewood--Paley square function} associated with the family $\Omega$ is defined as

$$
Tf(x):=\Big(\sum_{k\in \Z}|f\ast \check{{\bf 1}}_{\omega_k}(x)|^2\Big)^{1/2}. 
$$
Here $\check{f}$ denotes the inverse Fourier transform of $f$ defined as 
$$\check{f}(x):=\int_{\R} f(\xi) e^{2\pi i x\xi} d\xi,\quad f\in L^1(\R).$$
We briefly present the state of the matter concerning unweighted and weighted $L^p$-mapping properties related to this operator.

\subsection{Unweighted estimates for \textit{Littlewood--Paley square functions} }\label{subsec:un}

The unweighted $L^p$ boundedness of \textit{Littlewood--Paley square functions} is well understood in the linear case. Also, there have been some recent developments in the theory of bilinear Littlewood-Paley square functions, see~\cite{BF, Bern,BS} for instance.  Here we focus only on the linear square functions and discuss some of the main results in this direction.

If $\{\omega_k\}_{k\in \Z}$ is a \textit{lacunary} sequence of intervals, i.e., $\omega_k=[\lambda^k, \lambda^{k+1}]$ for some  $\lambda>1 $, then the classical theorem by Littlewood and Paley \cite{LP} is well-known, namely
\begin{equation*}
\label{un:lac}
\| T(f)\|_p \leq C_p \|f\|_p,\quad 1<p<\infty.
\end{equation*}
See also~\cite[Chapter 8, Section 8.2, p. 186--187]{Duobook} for details and references. The unweighted estimates for the lacunary (sometimes also called \textit{dyadic} when $\lambda=2$) square functions were further strengthened by J. Bourgain~\cite{B}, where he proved sharp asymptotic (in~$p$) unweighted estimates for the operator norm $\| T\|_p.$ More precisely, Bourgain proved the following. 
\begin{theorem}[\cite{B}] \label{un:sharpB} Let $T$ denote the lacunary Littlewood--Paley square function associated with the sequence of dyadic intervals. Then 
$$
\|T\|_p \simeq \frac{1}{(p-1)^{\frac{3}{2}}}\quad \text{as }p\rightarrow 1 \quad \text{ and }\quad  \|T\|_p \simeq p\quad\text{as}~p\rightarrow \infty.
$$
\end{theorem}
We would like to remark here that  Bourgain proved the above estimates for the Littlewood--Paley square functions in the Fourier series case on the circle group. However, as explained by A. Lerner in \cite{L}, these can be transferred to the real line version in a straightforward way using the transference principle. 
Recently, O. Bakas~\cite{Bak} gave an alternative proof for the first estimate in the above theorem.  We will discuss more about the sharp quantitative estimates for  square functions in Section~\ref{sec:Qweight}. 

The general case, namely the case of arbitrary disjoint intervals $\omega_k$, was addressed by J. L. Rubio de Francia in~\cite{Rubio}, extending earlier results of L. Carleson \cite{C} and A. C\'ordoba \cite{Cordoba}. By denoting this operator also by $T$, Rubio de Francia proved that 
\begin{equation}
\label{un:gen}
\| T(f)\|_p \leq C_p \|f\|_p,\quad 2\leq p<\infty.
\end{equation}
Moreover, the condition $p\geq 2$ is sharp in the above result, i.e.,  the operator $T$ fails to be bounded on $L^p$ for $p<2$ in general. 
This can be verified by considering the case when $\omega_k=[k,k+1]$, $k\in \Z$.  
The operator $T$ associated with a general family of intervals is commonly referred to as the \textit{Rubio de Francia Littlewood--Paley square function}. Rubio de Francia's result has been revisited many times and different proofs of \eqref{un:gen} were obtained with several techniques, among them we emphasize the techniques involving time-frequency analysis. The literature is extensive, we refer for instance to \cite{ Be,B0, TaoCowling,Journe,   Lacey1,Sjoelin, Soria} and references therein. 

Note that for $p=2$, the Plancherel theorem yields $\|T\|_2=1$. Therefore, at the end-point $p=2$, the question of sharp quantitative estimates for the operator norm $\|T\|_2$ is trivial. However, the question of asymptotic sharp quantitative estimates for $\|T\|_p$ as $p\rightarrow \infty$, is interesting. Again, this question has been addressed by Bourgain~\cite{B} in the case of circle group. He proved that 
\begin{equation*}\label{Qun:B} 
\|T\|_p \simeq p,\quad 2<p<\infty.
\end{equation*}

\subsection{Weighted estimates for \textit{Littlewood--Paley square functions} }

A weight $w$ is a nonnegative locally integrable function defined on $\R^n$.  Recall that given $1<p<\infty$, the Muckhenhoupt class of weights $A_p$ consists of all $w$ satisfying
$$
[w]_{A_p}:=\sup_Q \langle w\rangle_{Q}\big(\langle w^{1-p'}\rangle_{Q}\big)^{p-1}<\infty,
$$
where the supremum ranges over all cubes $Q$ in $\R^n$ with axes parallel to the coordinate axes. For $p=1,$ the class $A_1$ consists of all $w$ such that 
$$
[w]_{A_1}:=\essup\frac{M (w)}{w}<\infty.
$$
Here $M$ denotes the classical Hardy--Littlewood maximal function and we have used the notation 
$$
\langle w\rangle_{r,Q}:=\bigg(\frac{1}{|Q|}\int_Q |w|^r \bigg)^{\frac{1}{r}}, \quad \langle w\rangle_{Q}:=\langle w\rangle_{1,Q}.
$$
The constant $[w]_{A_p}, 1\leq p<\infty,$ is referred to as the $A_p$ characteristic of the weight $w$. We define, in a natural way, the $A_{\infty}$ class as $A_{\infty}=\cup_{p\ge 1}A_p$. Associated to this $A_{\infty}$ class it is also possible to define an $A_{\infty}$ constant as 
$$
[w]_{A_{\infty}}:=\sup_Q\frac{1}{w(Q)}\int_QM(w_{\chi_Q})\,dx.
$$
Here, $w(Q):=\int_Qw(x)\,dx$, and the supremum is taken over all cubes with edges parallel to the coordinate axes.

The weighted $L^p$ estimates for the dyadic Littlewood--Paley square functions were first proved by~Kurtz~\cite{Kurtz}. Recent developments in the theory of weights have focused on understanding sharp quantitative weighted estimates in terms of the $A_p$ characteristic for operators under consideration. The list of papers in this subject is vast and it would be a challenging task by itself to just write all of them down without a miss. Here we refer the reader to~\cite{BFP,CMP,CDpO,Hyt,HRT,Ler,dyadic} and references therein. We bring attention to the recent work of Lerner~\cite{L}, where he studied the lacunary square function and proved the following.
\begin{theorem}[\cite{L}] \label{qw:lerner}Let $T$ be the lacunary Littlewood--Paley square function. If $\alpha_p$ is the best possible exponent in the estimate
$$
\|T\|_{L^p(w)\to L^{p}(w)} \leq C_p [w]_{A_p}^{\alpha_p},
$$
then $$\max\Big (1, \frac{3}{2}\frac{1}{p-1}\Big)\leq \alpha_p\leq \frac{1}{2}\frac{1}{p-1} +\max\Big(1,\frac{1}{p-1}\Big),\quad 1<p<\infty.$$
In particular, $\alpha_p= \frac{3}{2}\frac{1}{p-1}$ is the best possible exponent in the range $1<p\leq 2$.
\end{theorem}

See also \cite{CDpLO}, where the Walsh--Fourier model for the lacunary square function is studied.

In the most general case, Rubio de Francia proved in \cite{Rubio} a weighted version of  \eqref{un:gen} with $A_{p/2}$ weights when $p>2$. However, it is not yet known whether this is true in $L^2(w)$ with $w\in A_1$, although he conjectured the validity of such an inequality in the seminal paper,
see Section~\ref{sec:Qweight}.

Motivated from the above works we address the question of sharp quantitative weighted estimates for the Rubio de Francia's square function in this article. We exploit the ideas presented in~\cite{Be, CDpO, L} in order to prove our results. In particular, we work with the model sum operator for the square function and prove sparse domination for that operator. 

\subsection{Main results} 

Time-frequency analysis in its current form was first developed by M. Lacey and C. Thiele~\cite{LT1,LT2} in their study of the bilinear Hilbert transform. Since then,  there have been extensions of these techniques to many different directions, with significant improvements. We would refer the reader to~\cite{MTT,Th} for a systematic account of this topic. Time-frequency analysis, based on stopping times and localizations, has been exploited by several authors, highlighting the helicoidal method developed by Benea and Muscalu (see \cite{BMs, BMh} and references therein). In particular, it is shown that the local estimates may be employed for proving sparse domination for several scalar operators in harmonic analysis. Another recent remarkable work was carried out by Culiuc et al. in \cite{CDpO}, where the authors formulate a pointwise domination principle for the class of multilinear multiplier operators invariant under modulations of the functions involved. Their proof is based on a stopping time construction based on localized outer-$L^p$ embedding theorems for the wave packet transform, see \cite{DO,DT}. The latter works lead to quantitative (and sometimes even new qualitative) weighted estimates.

The main purpose in this article is to provide quantitative weighted estimates for the Rubio de Francia's Littlewood--Paley square function. In order to do this, we will present a domination principle for the latter operator by positive sparse forms. This yields, in a standard way, quantitative weighted estimates for the operator. We get sharp exponent of the weighted characteristic in the range $3\leq p<\infty$. However, we do not know whether our estimates are sharp in the range of $2<p<3$.

 Our research relies on an adaptation of the time-frequency decomposition of the bilinear form associated with the Rubio de Francia's Littlewood-Paley operator shown by Benea in~\cite{Be} and the approach used in \cite{CDpO}. Moreover, the proofs are simpler than in \cite{CDpO}, since we do not need outer spaces and Carleson estimates. 
 
From now on we shall always work with the square function associated with a finite family of intervals $\Omega=\{w_k\}_{k=1}^N$ satisfying~(\ref{overl}). Since the boundedness results that we prove are independent of $N$, the corresponding results follow in the general case. 

Let $\vec{g}=\{g_k\}_{k=1}^N$ be a sequence of functions and consider the dual pair
\begin{equation*}
\label{eq:dualpair}
\langle Tf, \vec g\rangle =\sum_{k=1}^N \int_{\R} f\ast  \check{1}_{[a_k,b_k]}(x)g_k(x) dx.
\end{equation*} 
We shall use the notation $|\vec g(x)|:=\Big (\sum\limits_{k=1}^N |g_k(x)|^2\Big)^{\frac{1}{2}}$.

Let $\mathcal{D}$ be the standard system of dyadic intervals in $\R$, 
$$
\mathcal{D}:=\{2^{-k}([0,1)+m): k,m\in \Z\},
$$ 
consisting of dyadic half-open intervals of different scales $2^{-k}$, $k\in \Z$. We say that $\mathcal{S}$ is an $\eta$-sparse family ($0<\eta<1$) if for each $I\in \mathcal{S}$ there exists $E_I\subset I$ such that
\begin{enumerate} 
\item $\eta|I|\le |E_I|$.
\item The sets $E_I$ are pairwise disjoint. 
\end{enumerate}

The main result of this article, as described below, involves domination of the bilinear form associated with the Rubio de Francia's Littlewood-Paley operator by positive sparse forms.

\begin{theorem}
\label{thm:main}
Let $f,\vec g$ be in $C_0^{\infty}(\R)$. Then there exist a positive constant $K$ and a $\frac16$-sparse collection $\mathcal{S}$ such that
$$
|\langle Tf, \vec g\rangle| \le K \sum_{I\in \mathcal{S}} |I|\langle f\rangle_{2,I}\langle |\vec g|\rangle_{I}, 
$$
where the sparse collection $\mathcal{S}$ depends on $f$ and $\vec g$.   
\end{theorem}
Here in the above and in what follows, $\vec g\in C_0^{\infty}(\R)$ means that each $g_k$ is in $C_0^{\infty}(\R)$.
\begin{remark}\label{sharp:sparse}  We would like to remark here that the sparse domination in Theorem~\ref{thm:main} is sharp in the sense that the $L^2$-average $\langle f\rangle_{2,I}$ cannot be replaced by the $L^q$-average $\langle f\rangle_{q,I}$ with $q<2$, because it would imply strong $(p,p)$ boundedness of the Rubio de Francia's square function for $p<2,$ which is known to be invalid, as discussed previously  in Section~\ref{subsec:un}. 
\end{remark}
Theorem \ref{thm:main} implies the following quantitative weighted bounds for the operator $T$.

\begin{corollary}
 \label{cor:what0}
 For $2<p<\infty$ and any $w\in A_{p/2}$
 $$
 \|T\|_{L^p(w)\to L^{p}(w)}\lesssim [w]_{A_{p/2}}^{\max\big(\frac{1}{p-2},1\big)}.
 $$
  In particular, for $3\le p<\infty$, we have that $\max\big(\frac{1}{p-2},1\big)=1$ is the sharp exponent in this range.
 \end{corollary}

For $3\le p< \infty,$ the sharpness of the exponent $\alpha_p:=\max\big(\frac{1}{p-2},1\big)=1$ is immediate in view of  \cite[Theorem 5.2]{FN}. However we cannot guarantee the sharpness of $\alpha_p$ for $2< p<3$. Concerning this, we develop a discussion in Section \ref{sec:Qweight}. By \cite[Theorem 1.4]{FN} we also obtain, as an immediate corollary of the sparse domination, quantitative weighted weak type estimates at the end-point $p=2$. 
 \begin{corollary}
 \label{cor:whatw}
 For $w\in A_{1}$
 $$
 \|T\|_{L^2(w)\to L^{2,\infty}(w)}\lesssim [w]_{A_1}^{\frac{1}{2}}[w]_{A_{\infty}}^{\frac{1}{2}}\log(e+[w]_{A_{\infty}}).
 $$
 \end{corollary}
 We do not know whether the estimate above is sharp.
The qualitative version of the estimate in Corollary \ref{cor:whatw} was already proved in \cite[Theorem B (ii)]{K}.

The paper is organized as follows: In Section~\ref{sec:sparse} we introduce preliminary notion required for the time-frequency analysis and describe the model operator for the Rubio's square function. The proof establishing the sparse domination for the model operator, which is mainly the content of Theorem~\ref{thm:main2}, is given in Section~\ref{sec:proofm2}. Next, in Section~\ref{sec:Qweight} we discuss  the sharpness of exponents obtained in Corollary~\ref{cor:what0}. Finally, in the same section we develop a discussion about Rubio's conjecture. 

\section{Time-frequency analysis for Rubio's square function}
\label{sec:sparse}

We perform a time-frequency analysis of the bilinear form associated with the Rubio de Francia's Littlewood--Paley operator $\langle Tf, \vec g\rangle$. Most of the current section is standard and we borrowed its content from \cite{Be}, with the suitable modifications in our context. We also refer the reader to \cite{Lacey1, LT1, LT2, MTT, Th} for more information about the topic.

\begin{definition}[Tile] A \textit{tile} $P$ is a rectangle $P=I_P\times \omega_P$ of area one with the property that $I_P,\omega_P\in \mathcal{D}$ or $\omega_P$ is in a shifted variant of $\mathcal{D}$.
\end{definition}
\begin{definition}[Order relation on tiles] Given two tiles $P$ and $P'$ we say that $P<P'$ if $I_P\subsetneq I_{P'}$ and $\omega_{P'}\subset 3 \omega_{P}$. Further, we say that $P\leq P'$ if $P<P'$ or $P=P'$. 
\end{definition}

For any interval $I\subset \R$, define the smooth localized variant of the characteristic function ${\bf 1}_I$ by  
\begin{equation}
\label{eq:chi}
\widetilde{\chi}_I(x):=\Big(1+\frac{\dist(x,I)}{|I|}\Big)^{-100}.
\end{equation}
\begin{definition}[$L^1$-normalized wave packets]
Let $P=I_P\times \omega_P$ be a tile. An $L^1$- normalized wave packet on $P$ is a smooth function $\phi_P$ which has Fourier support in the frequency interval $\omega_P$ and is $L^1$-adapted to the time interval $I_P$ in the sense that  
$$
|\phi_P^{(n)}(x)|\le C_{n,M}\frac{1}{{|I_P|^{n+1}}}\frac{1}{\big(1+\frac{\dist(x,I_P)}{|I_P|}\big)^M}
$$
for sufficiently many derivatives $n$, and a large number $M$.
\end{definition}
Let $ \mathbb{W}_{k}$ denote the Whitney decomposition of $w_k=[a_k,b_k]$ with respect to its endpoints. More precisely, each $J\in  \mathbb{W}_{k}$ is the maximal dyadic interval contained in $[a_k,b_k]$ with the property that $\operatorname{dist}(J,a_k)\ge |J|$ and $\operatorname{dist}(J,b_k)\ge |J|$. We need to consider tiles which have frequency intervals associated with the collection of intervals $\{w_k\}_{k=1}^N$. Therefore, for each 
$k,~1\leq k\leq N$, consider $\po_k$ to be the collection of tiles $P=I_P\times \omega_P$ such that $\omega_P \in \mathbb{W}_{k}$. 
The collection $\po=\cup_{k=1}^N \po_k$ is the complete collection of tiles which will play a role in the time-frequency decomposition of the operator. 

We can write the collection $\po_k$ as a frequency translate of a single collection in the following sense (see \cite[Remark 22]{Be}). Consider the frequency interval of reference $[0,L]$, where $L:=\max\limits_{1\leq k\leq N}|a_k-b_k|$. Let $\po_0$ be the collection of tiles associated to this interval consisting of tiles of the form 
$$
P=I\times \omega, \quad \text{ where } I \text{ is a dyadic interval, and } \omega=\Big[\frac{1}{2|I|},\frac{1}{|I|}\Big].
$$
Then for any $1\le k\le N$, the collection $\po_k$ can be written as the frequency translate 
 $\po_k=(\po_0+\nu_k)\cap \po_k$. This representation would be helpful in considering the vectorial tree structure on the collection of tiles under consideration.  

From now on, we shall assume that the collection of tiles $\po_k$ is finite, while the frequencies of the intervals are lacunary with respect to the $a_k$'s only. 
Following \cite{OSTTW}, one can write
$$
{\bf 1}_{[a_k,b_k]}(\xi)\widehat{f}(\xi)=\sum_{P\in \po_k}|I_P| \langle f,\phi_{P}\rangle \widehat{\phi_{P}}(\xi).
$$
The model bilinear form associated to the Rubio de Francia square function is given by 
$$
\Lambda_{\po} (f,\vec g)=\langle Tf, \vec g\rangle =\sum_{k=1}^N \sum_{P\in \po_k}|I_P|\langle f,\phi_{P}\rangle \langle g_k, \phi_{P}\rangle.
$$

We shall prove estimates for the bilinear form (with same notation) with absolute values of coefficients, i.e. 
$$
\Lambda_{\po} (f,\vec g)=\sum_{k=1}^N \sum_{P\in \po_k}|I_P||\langle f,\phi_{P}\rangle| |\langle g_k, \phi_{P}\rangle|.
$$ 
This does not cause any additional difficulty because this may be thought of as a new bilinear form with $g_k$ replaced by $\epsilon_k g_k$ for  $\epsilon_k\in \{\pm 1\}$.  Further, for a subcollection of tiles $\qo\subseteq \po$, the associated bilinear form may be defined in a standard way and will be denoted by $\Lambda_{\qo} (f,\vec g)$.  

In order to prove Theorem \ref{thm:main}, it suffices to prove the following result, whose proof is given in Section \ref{sec:proofm2}.
\begin{theorem}
\label{thm:main2}
Let $f,\vec g$ be $ C_0^{\infty}(\R)$ functions. Then there exist a positive constant $K$ and a $\frac16$-sparse collection $\mathcal{S}$ such that
$$
\Lambda_{\po}(f,\vec g)\le K \sum_{I\in \mathcal{S}} |I|\langle f\rangle_{2,I}\langle |\vec g|\rangle_{I}.
$$
where the sparse collection $\mathcal{S}$ depends on $f$ and $\vec g$. 
\end{theorem}
Next, we recall several definitions in order to perform time-frequency analysis for the model operator $\Lambda_{\po}(f,\vec g)$ (see \cite{Be} for more details). 
\begin{definition}
A subcollection $T$ of tiles is called a \textit{tree} with \textit{top} $P_T$ if there exists a tile $P_T=I_T\times \omega_T$ and a frequency point $\xi_T\in \omega_T$ with the property that 
$$
P\leq P_T, \quad \text{ and } \xi_T\in 7\omega_P\quad \text{ for every } P\in T.
$$
\end{definition}

\begin{definition}
Let $\{\po_k\}_k$ be as above. We say that $\vec{T}\subset \po=\bigcup_k \po_k$ is a \textit{vectorial tree} if there exists a tree $T_0\subset \po_0$ so that for every $1\le k\le N$,
$$
T_k:=\vec{T}\cap \po_k=\nu_k+T_0\quad \text{ is a frequency translation of } T_0.
$$
\end{definition}

The classical ``sizes'' are defined as follows. 
\begin{definition}[Vectorial size]
$$
\overrightarrow{\operatorname{size}}_{\po}(f):=\sup_{\substack{\vec{T}\subset \po\\
\textrm{ vectorial tree }}}\Big(\frac{1}{|I_T|}\sum_{k=1}^N\sum_{P\in T_k}|I_P||\langle f, \phi_P\rangle|^2\Big)^{1/2},
$$
where the supremum is taken over vectorial trees 
$\vec{T}\subset \po$.
\end{definition}

\begin{definition}[Dual size]
\label{def:dual}
$$
\widetilde{\operatorname{size}}_{\po}(g):=\sup_{I'\in \mathcal{J}^+_{\po}(I_T)}\frac{1}{|3I'|}\int_\R|g(x)|\widetilde{\chi}_{3I'}(x)\,dx,
$$
where $\widetilde{\chi}_I$ was defined in \eqref{eq:chi}. Here, given a collection $\po$ of tiles, we denote by $\mathcal{J}_{\po}^+$ the collection of dyadic intervals $I'$ which contain some $I_P$, with $P\in \po$. 
\end{definition}
We have the following estimate for vectorial size.

\begin{proposition}[\cite{Be} Proposition 24] \label{size} For $f\in L^2_{\it {loc}}$, we have 
$$
\overrightarrow{\operatorname{size}}_{\po}(f)\lesssim \sup_{P\in \po}\left(\frac{1}{|I_P|}\int_\R|f(x)|^2\widetilde{\chi}^M_{I_P}(x)\,dx \right)^{\frac{1}{2}},
$$
for some $M\ge1$.
\end{proposition}
Lemma \ref{lem:local} below contains estimates on single vectorial trees. The proof of this lemma follows by some modifications in the proof of \cite[Lemma 12]{Be}, see also \cite{MTTI, MTT}, hence we omit the details. Indeed, the proof may be obtained following the steps of the proof of \cite[Lemma 12]{Be} along with the duality argument for the mixed norm. 

\begin{lemma}[Localization Lemma/Single vectorial tree estimate]
\label{lem:local}
Let $\vec{T}\subset \po$ be a vectorial tree. Let $\vec{g}=\{g_k\}$ be an $\ell^2$ sequence of functions as earlier. Then
$$
\Lambda_{\vec{T}}(f,\vec{g})\lesssim \ds(|\vec{g}|)\cdot \vs(f)\cdot |I_T|.
$$
\end{lemma}

Next, we recall the stopping time algorithms (see \cite[Pages 143--144]{Be}), which play a crucial role in the proof. 

\begin{lemma} [Decomposition lemma for vectorial size] \label{dvs}
Let $\po$ be a collection of tiles such that  $\overrightarrow{\operatorname{size}}_{\po}(f)\leq \lambda.$ Then one can decompose $\po=\po'\cup \po''$ with $\overrightarrow{\operatorname{size}}_{\po'}(f)\leq \frac{\lambda}{2}$ and $\po''$ can be written as a union of disjoint vectorial trees $\po''=\cup_{\vec{T}} {\vec{T}}$ such that $\sum\limits_{\vec{T}} |I_T|\lesssim \lambda^{-2} \|f\|_2^2.$
\end{lemma}
Let us iterate the above lemma. We have $\po=\po'\cup \po''$. 
Write $\po_1:=\po''$, then we have $\overrightarrow{\operatorname{size}}_{\po_1}(f)\leq \overrightarrow{\operatorname{size}}_{\po}(f)$ (here we are using the fact that the size of a subcollection is smaller than or equal to the original collection) and $\po_1=\cup_{\vec{T}} {\vec{T}}$ such that $\sum\limits_{\vec{T}} |I_T|\lesssim \lambda^{-2} \|f\|_2^2.$ Next, we perform the same decomposition to the collection $\po=\po'$. This gives us $\po'=(\po')'\cup (\po')''$ with $\overrightarrow{\operatorname{size}}_{(\po')'}(f)\leq \overrightarrow{\operatorname{size}}_{\po'}(f)\leq \lambda/2$ and $(\po')''=\cup_{\vec{T}} {\vec{T}}$ such that $\sum\limits_{\vec{T}} |I_T|\lesssim (\lambda/2)^{-2} \|f\|_2^2.$ Note that this time $\overrightarrow{\operatorname{size}}_{(\po')''}(f)\leq \overrightarrow{\operatorname{size}}_{\po'}(f)\leq \lambda/2$. We rename $(\po')''=:\po_2$ and continue the iterative procedure with $(\po')'$ and so on. This will give us the collection $\po_n$ with negative powers of $2$ along with $\lambda$, i.e. $\lambda/2^n$.

With the reasoning above, we obtain the following corollary.
\begin{corollary} 
\label{corenergy}
Let $\po$ be a collection of tiles. Then one can decompose $\po=\cup_n \po_n$ such that 
$$
\overrightarrow{\operatorname{size}}_{\po_n}(f)\leq \min{(\lambda 2^{-n}, \overrightarrow{\operatorname{size}}_{\po}(f))}
$$ and 
$$\po_n=\cup_{\vec{T}\in \mathbb{T}_n} {\vec{T}}\quad\text{with}\quad \sum\limits_{\vec{T}\in \mathbb{T}_n} |I_T|\lesssim \frac{2^{2n}}{\lambda^2} \|f\|_2^2.$$
\end{corollary}
A similar decomposition lemma may be proved for the other size. 

\begin{lemma} [Decomposition lemma for dual size] \label{dds}
Let $\po$ be a collection of tiles such that  $\widetilde{\operatorname{size}}_{\po}(g)\leq \lambda.$ Then one can decompose $\po=\po'\cup \po''$ with $\widetilde{\operatorname{size}}_{\po'}(g)\leq \frac{\lambda}{2}$ and $\po''$ can be written as a union of disjoint vectorial trees $\po''=\cup_{\vec{T}} {\vec{T}}$ such that $\sum\limits_{\vec{T}} |I_T|\lesssim \lambda^{-1} \|g\|_1$.
\end{lemma}
\begin{corollary} 
\label{cormass}
Let $\po$ be a collection of tiles. Then one can decompose $\po=\cup_n \po_n$ such that 
$$
\widetilde{\operatorname{size}}_{\po_n}(g)\leq \min{(\lambda2^{-n}, \widetilde{\operatorname{size}}_{\po}(g)})
$$ 
and 
$$\po_n=\cup_{\vec{T}\in\mathbb T_n} {\vec{T}}\quad \text{with}\quad  \sum\limits_{\vec{T}\in \mathbb T_n} |I_T|\lesssim \frac{2^n}{\lambda} \|g\|_1.$$
\end{corollary}

\section{Proof of Theorem \ref{thm:main2}}
\label{sec:proofm2}

We proceed to prove Theorem \ref{thm:main2} and as explained earlier, this will lead to the conclusion of Theorem \ref{thm:main}.

\subsection{Sparse collection and reduction to a single grid}
\label{sub:sparse}

First we will construct the sparse collection $\mathcal{S}$. Let $f\in L^p(\R)$ and write 
$$
M_pf(x):=\sup_{I\subset \R} \langle f\rangle_{p,I}{\bf 1}_{I}(x)
$$
for the $p$-Hardy--Littlewood maximal functions. When $p=1$ we will sometimes omit the subindex and denote it by $M$.

For a fixed $Q\in  \mathcal{D}$, the $p$-stopping intervals of $f$ on $Q$, which we will denote by $\mathcal{I}_{f,p,Q}$, are defined as the collection of maximal dyadic $I\subset Q$ such that
\begin{equation}
\label{eq:maxic}
 I\subset \{x\in \R: M_p(f{\bf 1}_{3Q})(x)\ge C \langle f\rangle_{p,3Q}\}.
\end{equation}
Observe that $
\mathcal{I}_{f,p,Q}$ is a pairwise disjoint collection of dyadic intervals. Furthermore, we have sparsity due to the maximality, namely 
$$
\sum_{I\in 
\mathcal{I}_{f,p,Q}}|I|\le  |\{x\in \R: M_p(f{\bf 1}_{3Q})\ge C \langle f\rangle_{p,3Q}\}|\le \frac{|Q|}{6}
$$ 
for $C$ large enough.

Given two compactly supported functions $f_j\in L^{p_j}(\R)$, $j=1,2$, we define, for all $Q\in \mathcal{D}$,
$$
\mathcal{I}_{(f_1,f_2),(p_1,p_2),Q}:=\textrm{ maximal elements of } \bigcup_{j=1}^2\mathcal{I}_{f_j,p_j,Q}.
$$
Then the intervals $\mathcal{I}_{(f_1,f_2),(p_1,p_2),Q}$ are pairwise disjoint and 
\begin{equation}
\label{eq:packi}
\sum_{I\in \mathcal{I}_{(f_1,f_2),(p_1,p_2),Q}}|I|\le \frac{|Q|}{2}.
\end{equation}
From the definition of $\mathcal{I}_{(f_1,f_2),(p_1,p_2),Q}$, there holds
\begin{equation}
\label{eq:maxi}
\inf_{x\in 3I} M_{p_j}(f_j{\bf 1}_{3Q})(x)\lesssim \langle f_j\rangle_{p,3Q}, \quad \text{ for all } I\in \mathcal{I}_{(f_1,f_2),(p_1,p_2),Q},\, \, j=1,2.
\end{equation}

The procedure to construct the sparse collection follows an inductive argument as in \cite[Section 5]{CDpO}. We begin by choosing a partition of $\R$ by intervals 
$$
\{Q_k\in \mathcal{D}: k\in \N\}
$$ 
with the property that $\supp f_j\subset 3Q_k$ for all $j=1,2$ and $k\in \N$. For each $k$, let $
\mathcal{S}(Q_k)=\bigcup_{\ell=0}^{\infty}\mathcal{S}_{\ell}(Q_k)$ 
where $\mathcal{S}_0(Q_k)=\{Q_k\}$ and, proceeding iteratively, take 
$$
\mathcal{S}_{\ell}(Q_k)=\bigcup_{Q\in \mathcal{S}_{\ell-1}(Q_k)} \mathcal{I}_{(f_1,f_2),(p_1,p_2),Q}, \quad \ell =1,2,\ldots.
$$
Finally, define $
\mathcal{S}=\mathcal{S}(\mathcal{D}, f_1,f_2)=\bigcup_{k=0}^{\infty}\mathcal{S}(Q_k)$.
By construction and by the packing property \eqref{eq:packi}, $\mathcal{S}$ is a $\frac12$-sparse subcollection of $\mathcal{D}$.

Let us consider the three canonical shifted grids on $\R$, namely
$$
\mathcal{D}_j=\{2^k[0,1)+\big(n+\frac{j}{3}\big)2^k:k,n\in \Z\},\quad  j=0,1,2.
$$ 
It is well-known that for all intervals $I\subset\R$ there exists a unique $\widetilde{I}\in \mathcal{D}_0\cup\mathcal{D}_1\cup\mathcal{D}_2$ with $3I\subset \widetilde{I}$, $|\widetilde{I}|\le 6\cdot |3I|$. We say that $I$ has type $j\in \{0,1,2\}$ if $\widetilde{I}\in \mathcal{D}_j$. 
Fix a finite collection of tiles $\po$ and a pair of functions $(f_1,f_2)$. Let us split $\po=\po^0\cup\po^1\cup \po^2$ where $\po^j=\{P\in \po: I_P \text{ has type } j\}$. 

With all these ingredients at hand, let us explain the strategy in our context, in which we will consider $f$ and $\vec g$. For each $j\in \{0,1,2\}$ we will use the previous construction with $\mathcal{D}=\mathcal{D}_j$ to obtain a $\frac12$-sparse collection of intervals $\mathcal{S}_j=\mathcal{S}(\mathcal{D}_j,f,|\vec g|)$ such that 
\begin{equation}
\label{eq:sparseJ}
\Lambda_{\po^j}(f,\vec g)\lesssim K \sum_{I\in \mathcal{S}_j} |I|\langle f\rangle_{2,I}\langle |\vec g|\rangle_{I}
\end{equation}
for suitable large $K$. Once we obtain \eqref{eq:sparseJ}, we will conclude the estimate
$$
\Lambda_{\po}(f,\vec g)\lesssim \sum_{j=0}^2\Lambda_{\po^j}(f,\vec g)\le K \sum_{j=0}^2\sum_{I\in \mathcal{S}_j} |I|\langle f\rangle_{2,I}\langle |\vec g|\rangle_{I}\lesssim K\sum_{I\in \widetilde{S}} |I|\langle f\rangle_{2,I}\langle |\vec g|\rangle_{I},
$$
where $\widetilde{S}=\{3I:I\in \mathcal{S}_{j_0}\}$ and $j_0\in \{0,1,2\}$ is such that the right hand side of \eqref{eq:sparseJ} is maximal. Since $\mathcal{S}_{j_0}$ is $\frac12$-sparse it immediately follows that $\mathcal{S}$ is a $\frac16$-sparse collection. This would complete the proof of Theorem \ref{thm:main2}. So it suffices to prove \eqref{eq:sparseJ}. Moreover, we can work just with $j=0$, so will omit the subscript $j$ from now on.

We will prove \eqref{eq:sparseJ} after several steps. 

\subsection{The good tiles}

For a collection of tiles $\po$, consider the following subcollections
$$
\po_{\le}(I):=\{P\in \po; I_P\subset I\}, \qquad \po_{=}(I):=\{P\in \po; I_P=I\}.
$$
Fix a finite collection of tiles $\po$ whose intervals $\{I_P: P\in \po\}$ are dyadic. For $\po_{\le}(I)$, define the set of \textit{good tiles} as 
$$
\mathrm{G}_{(f,|\vec g|),(p_1,p_2),Q}:=\po\setminus \Big(\bigcup_{I\in \mathcal{I}_{(f,|\vec g|),(p_1,p_2),Q}}\po_{\le}(I)\Big).
$$
We may define the good tiles for any collection of tiles $\po$. When we use it, it will be with respect to $\po_Q$ for intervals $Q$, as we are going to work with functions that are supported in $3Q$ for some $Q$. 

We will prove the following proposition.
\begin{proposition}
\label{prop:claim1}
Let $Q\subset \R$ be a dyadic interval and $f,\vec g\in C_0^{\infty}(\R)$ with $\supp f, \supp |\vec g| \subset 3Q$. Then
\begin{equation*}
\label{eq:1}
\sum_{k=1}^N\sum_{P\in \mathrm{G}_{(f,|\vec g|), (2,1),Q}\cap \po_k}|I_P||\langle f,\phi_P\rangle| |\langle g_k,\phi_P\rangle| \lesssim |Q| \langle f\rangle_{2,3Q}\,\langle |\vec g|\rangle_{3Q}.
\end{equation*}
\end{proposition}

We need some preparation to prove Proposition \ref{prop:claim1}.
Fix $Q\in \mathcal{D}$. Let us define now, for a collection of tiles $\po$ whose intervals $\{I_P:P\in \po\}$ are dyadic, 
$$
\mathrm{G}_{f, 2,Q}=\po\setminus \Big(\bigcup_{I\in \mathcal{I}_{f,2,Q}}\po_{\le}(I)\Big).
$$ 
Similarly, define the collection 
$$
\mathrm{G}_{|\vec g|, 1,Q}=\po\setminus \Big(\bigcup_{I\in \mathcal{I}_{|\vec g|,1,Q}}\po_{\le}(I)\Big).
$$ 

\begin{lemma}
\label{lem:vsaver}
Let $\mathrm{G}_{f, 2,Q}$ as above and let $f$ be such that $\supp(f)\subseteq 3Q$. Then
$$
\overrightarrow{\operatorname{size}}_{\mathrm{G}_{f, 2,Q}}(f)\lesssim \langle f\rangle_{2,3Q}.
$$
\end{lemma}
\begin{proof}
Recall Proposition~\ref{size} and observe that it is enough to prove that for $P\in \mathrm{G}_{f, 2,Q}$ we have 
\begin{equation}
\label{eq:max2}
\Big(\frac{1}{|I_P|}\int_{\R}|f|^2\widetilde{\chi}_{I_P}(x)\,dx\Big)^{1/2}\le \inf_{x\in I_P}M_2f(x). 
\end{equation}
Once \eqref{eq:max2} is proven, by the definition of good tiles for $P\in \mathrm{G}_{f, 2,Q}$ we have that 
$$
I_P\cap \{x\in \R: M_2(f1_{3Q})\ge C \langle f\rangle_{2,3Q}\}=\emptyset
$$ 
and as a consequence, $M_2(f1_{3Q}) \leq C  \langle f\rangle_{2,3Q}$ on $I_P$. 

Let us prove \eqref{eq:max2}:
\begin{align*}
&\Big(\frac{1}{|I_P|}\int_{\R}|f|^2\widetilde{\chi}_{I_P}(x)\,dx\Big)^{1/2}\\
&\qquad=\Big(\frac{1}{|I_P|}\int_{I_P}|f|^2\,dx+\frac{1}{|I_P|}\sum_{k\ge 0}\int_{2^{k+1}I_P\setminus 2^{k}I_P}|f|^2\Big(1+\frac{\dist(x,I_P)}{|I_P|}\Big)^{-100}\,dx\Big)^{1/2}\\
&\qquad \lesssim\Big(\frac{1}{|I_P|}\int_{I_P}|f|^2\,dx+\frac{1}{|I_P|}\sum_{k\ge 0}\int_{2^{k+1}I_P\setminus 2^{k}I_P}|f|^2\Big(1+\frac{2^k|I_P|}{|I_P|}\Big)^{-100}\,dx\Big)^{1/2}\\
&\qquad \lesssim \Big(\frac{1}{|I_P|}\int_{I_P}|f|^2\,dx+\sum_{k\ge 0}2^{-100k}\frac{2^{k+1}}{{|I_P|2^{k+1}}}\int_{2^{k+1}I_P}|f|^2\,dx\Big)^{1/2},
\end{align*}
and the result follows easily.
\end{proof}
\begin{lemma}
\label{lem:dsaver}
Let $\mathrm{G}_{|\vec g|,1,Q}$ be as above, then
$$
\widetilde{\operatorname{size}}_{\mathrm{G}_{|\vec g|,1,Q}}(|\vec g|)\lesssim \langle |\vec g|\rangle_{3Q}.
$$
\end{lemma}
The proof is similar to the one in previous Lemma \ref{lem:vsaver} for vectorial size. 

\begin{proof}[Proof of Proposition \ref{prop:claim1}]

Recall that for vectorial tree $\vec{T}\subset \mathrm{G}_{(f,|\vec g|), (2,1),Q}$, by Lemma \ref{lem:local}, we have the estimate 
$$
\Lambda_{\vec{T}}(f,\vec g)\lesssim \ds(|\vec g|)\cdot \vs(f)\cdot |I_T|.
$$
Let 
$$
\mathrm{G}_{(f,|\vec g|),(2,1),Q}:=\po_{\le }(Q)\setminus \Big(\bigcup_{I\in \mathcal{I}_{(f,|\vec g|),(2,1),Q}}\po_{\le}(I)\Big).
$$
Note that $\mathrm{G}:=\mathrm{G}_{(f,|\vec g|), (2,1),Q}$ is possibly a smaller  collection than collections $\mathrm{G}_{f, 2,Q}$ and $\mathrm{G}_{|\vec g|,1,Q}$ and hence the estimates proved in Lemma~\ref{lem:vsaver} and \ref{lem:dsaver} are valid for this collection. Therefore we shall use, for the collection $\po=\mathrm{G}$, the single tree estimate proved in Lemma \ref{lem:local}, Lemmas~\ref{lem:vsaver} and \ref{lem:dsaver}, the decomposition Lemmas~\ref{dvs} and \ref{dds}  with $\lambda=\lambda_1= \langle f\rangle_{2,3Q}$ and $\lambda=\lambda_2= \langle g\rangle_{3Q}$ respectively, and Corollaries \ref{corenergy} and \ref{cormass} to get the result. Indeed, given $\theta_1+\theta_2=1$, we have
\begin{align}
\label{eq:cuenta}
\notag 
\Lambda_{\mathrm{G}}(f,\vec g)&\lesssim \sum_{n_1\in \N}\sum_{n_2\in \N}\sum_{\vec{T}\in \mathbb{T}_{n_1}\cap \mathbb{T}_{n_2}}\Lambda_{\vec{T}\in \mathbb{T}_{n_1}\cap \mathbb{T}_{n_2}}(f,\vec g)\\
\notag&\overset{\text{Lemma } \ref{lem:local}}{\lesssim}\sum_{n_1}\sum_{n_2}\sum_{\vec{T}}\overrightarrow{\operatorname{size}}_{\vec{T}\in \po_{n_1}}(f)\cdot\widetilde{\operatorname{size}}_{\vec{T}\in \po_{n_2}}(|\vec{g}|)\cdot |I_T|\\
\notag&\overset{\text{Cor. }\ref{corenergy} \text{ and }\ref{cormass} }{\lesssim}\sum_{n_1}\sum_{n_2}\sum_{\vec{T}}\min{(\lambda_12^{-n_1}, \overrightarrow{\operatorname{size}}_{\po}(f)})\min{(\lambda_22^{-n_2}, \widetilde{\operatorname{size}}_{\po}(|\vec{g}|)})|I_T|\\
\notag&\overset{\text{Lemmas } \ref{lem:vsaver} \text{ and } \ref{lem:dsaver}}{\lesssim} \sum_{n_1}\sum_{n_2}2^{-n_1}\lambda_12^{-n_2}\lambda_2\sum_{\vec{T}}|I_T|^{\theta_1+\theta_2}\\
\notag&\lesssim \sum_{n_1}\sum_{n_2}2^{-n_1}\lambda_12^{-n_2}\lambda_2\big(\sum_{\vec{T}}|I_T|\big)^{\theta_1}\big(\sum_{\vec{T}}|I_T|\big)^{\theta_2}\\
\notag&\overset{\text{Cor. }\ref{corenergy} \text{ and }\ref{cormass} }{\lesssim} \sum_{n_1}\sum_{n_2}2^{-n_1}\lambda_1\big(2^{2n_1}\lambda_1^{-2}\|f1_{3Q}\|_2^2\big)^{\theta_1}2^{-n_2}\lambda_2\big(2^{n_2}\lambda_2^{-1}\||\vec g|1_{3Q}\|_1\big)^{\theta_2}\\
&=\sum_{n_1}\sum_{n_2}2^{-n_1(1-2\theta_1)}2^{-n_2(1-\theta_2)}\lambda_1^{1-2\theta_1}\lambda_2^{1-\theta_2}\|f1_{3Q}\|_2^{2\theta_1}\||\vec g|1_{3Q}\|_1^{\theta_2}.
\end{align}
First observe that $\lambda_1^{1-2\theta_1}=\langle f\rangle_{2,3Q}\frac{1}{|3Q|^{-\theta_1}}\big(\int_{3Q}|f|^2\big)^{-\theta_1}$,
so
\begin{equation}
\label{eq:lambda1}
\lambda_1^{1-2\theta_1}\|f1_{3Q}\|_2^{2\theta_1}=\langle f\rangle_{2,3Q}|3Q|^{\theta_1}.
\end{equation}
Secondly, 
\begin{equation}
\label{eq:lambda2}
\lambda_2^{1-\theta_2}\||\vec g|1_{3Q}\|_1^{\theta_2}=\frac{1}{|3Q|^{1-\theta_2}}\int_{3Q}|\vec g|=\langle |\vec g|\rangle_{3Q}|3Q|^{\theta_2}.
\end{equation}
Hence, plugging \eqref{eq:lambda1} and \eqref{eq:lambda2} into \eqref{eq:cuenta}, we obtain
$$
\Lambda_{\mathrm{G}}(f,\vec g)\le \sum_{n_1}\sum_{n_2}2^{-n_1(1-2\theta_1)}2^{-n_2(1-\theta_2)}\langle f\rangle_{2,3Q}\langle |\vec g|\rangle_{3Q}|3Q|,
$$
and it is enough to choose $\theta_1<1/2$. The proof is complete.
\end{proof}

\subsection{Proof of \eqref{eq:sparseJ}}

Let $\{Q_k:k\in \N\}$ be the intervals used in the construction of $\mathcal{S}$ in Subsection \ref{sub:sparse}. Observe that $\{Q_k:k\in \N\}$ is a partition of $\R$, so we have the splitting
$$
\po=\bigcup_{k=0}^{\infty}\po_{\le}(Q_k).
$$
Since the collection $\po$ is finite, then this union is also finite. Moreover, $\mathcal{S}=\cup_k\mathcal{S}(Q_k)$, thus the estimate \eqref{eq:sparseJ} (and hence Theorem \ref{thm:main2}, as explained above) is a consequence of 
\begin{equation}
\label{eq:newes}
\Lambda_{\po_{\le}(Q_k)}(f,\vec g)\le K\sum_{Q\in \mathcal{S}(Q_k)}|3Q|\langle f\rangle_{2,3Q}\langle |\vec g|\rangle_{3Q}.
\end{equation}
In its turn, \eqref{eq:newes} is obtained by iteration of the lemma below, starting with $Q=Q_k$, which is valid because $\supp f, \supp |\vec g|\subset 3Q_k$. 
\begin{lemma}
\label{lem:recur}
Let $f,\vec g$ be as above and $Q\in \mathcal{D}$. For a collection of tiles $\po$ such that $\{I_P:P\in \po\}\subset \mathcal{D}$, there holds
$$
\Lambda_{\po_{\le}(Q)}(f{\bf 1}_{3Q},\vec g{\bf 1}_{3Q})\le K|3Q|\langle f\rangle_{2,3Q}\langle |\vec g|\rangle_{3Q}+\sum_{I\in \mathcal{I}_{(f,|\vec g|),(2,1),Q}}\Lambda_{\po\le (I)} (f1_{3I},\vec g 1_{3I}).
$$
\end{lemma}
Since $\po_{\le}(Q_k)$ is finite, the collections $\po_{\le}(I)$ will be empty after a finite number of iterations, at which point the iterative procedure leading to \eqref{eq:newes} is complete. Let us describe then how to show Lemma \ref{lem:recur}.

For the sake of brevity, let us assume that the functions $f$ and $|\vec g|$ are supported on $3Q$. We decompose 
$$
\Lambda_{\po_{\le}(Q)}(f,\vec{g})\le \sum_{k=1}^{N}\sum_{P\in \mathrm{G}_{(f,|\vec g|), (2,1),Q}\cap \po_k}|I_P||\langle f,\phi_P\rangle| |\langle g_k,\phi_P\rangle|+\sum_{I\in \mathcal{I}_{(f,|\vec g|),(2,1),Q}}\Lambda_{\po\le (I)} (f,\vec g).
$$
The first term satisfies the estimate in Proposition \ref{prop:claim1}.
Concerning the second one, we will prove the following recursive estimate. 

\begin{lemma}
\label{lem:claim2}
Let $f,\vec g$ as above and $Q\in \mathcal{D}$. For a collection of tiles $\po$ such that $\{I_P:P\in \po\}\subset \mathcal{D}$, there holds
\begin{equation}
\label{eq:claim2}
\sum_{I\in \mathcal{I}_{(f,|\vec g|),(2,1),Q}}\Lambda_{\po\le (I)} (f,\vec g)\lesssim K|Q| \langle f\rangle_{2,3Q}\,\langle |\vec g|\rangle_{3Q}+\sum_{I\in \mathcal{I}_{(f,|\vec g|),(2,1),Q}}\Lambda_{\po\le (I)} (f{\bf 1}_{3I},\vec g{\bf 1}_{3I}).
\end{equation}
\end{lemma}

For the proof of Lemma \ref{lem:claim2} we need some more ingredients. For each $I\in \mathcal{I}_{(f,|\vec g|),(2,1),Q}$, define
$$
\Lambda_{\po\le (I)}^{\vec{t}}(f,\vec g):=\Lambda_{\po\le (I)}(f{\bf 1}_{I^{t_1}},\vec g{\bf 1}_{I^{t_2}}):=\sum_{k=1}^{N}\sum_{P\in \po_{k \,\le}(I)}|I_P||\langle f{\bf 1}_{I^{t_1}},\phi_P\rangle| |\langle g_k {\bf 1}_{I^{t_2}},\phi_P\rangle|,
$$
where $\vec{t}=(t_1,t_2)\in \{\operatorname{in},\operatorname{out}\}^2$ and $I^{\operatorname{in}}:=3I$, $I^{\operatorname{out}}:=\R\setminus 3I$. We split now
$$
\Lambda_{\po_{\le}(I)}(f,\vec g)\le \sum_{\vec{t}\in  \{\operatorname{in},\operatorname{out}\}^2}\Lambda_{\po_{\le}(I)}^{\vec{t}}(f,\vec g).
$$
Observe that the term concerning $ (\operatorname{in},\operatorname{in})$ corresponds to the second term on the right hand side of \eqref{eq:claim2}. Then, we need to bound the terms $\Lambda_{\po_{\le}(I)}^{\vec{t}}$ such that $t_j=\operatorname{out}$ for at least one $j=1,2$. This is the content of the next lemma (which is analogous to \cite[Proposition 5.2]{CDpO}).

\begin{lemma}
\label{lem:claim3}
Let us assume that $\vec{t}$ is such that $t_j=\operatorname{out}$ for at least one $j=1,2$. Then
$$
\Lambda_{\po_{\le}(I)}^{\vec{t}}(f,\vec g)\lesssim |I|\inf_{x\in 3I}M_2f(x)\inf_{x\in 3I}M_1|\vec g|(x).
$$
\end{lemma}
\begin{proof}
Let us assume without loss of generality that $t_2=\operatorname{out}$.
First, we will prove the following: Let $J$ be an interval and $\supp |\vec g|\,\cap AJ=\emptyset$, with $A\ge3$. Then
\begin{equation}
\label{eq:53}
\Lambda_{\po_{=}(J)}(f,\vec{g}):=\sum_{k=1}^{N}\sum_{P\in \po_{k \,=}(J)}|I_P||\langle f,\phi_P\rangle| |\langle g_k,\phi_P\rangle|\lesssim A^{-98}|J|\inf_{x\in 3J}M_2f(x)\inf_{x\in 3J}M_1|\vec g|(x).
\end{equation}
Indeed, note that the collection $\po_{=}(J)$ is a vectorial tree with top (time interval) $I_J$. Then by Lemma \ref{lem:local}, Proposition \ref{size} and Definition \ref{def:dual},
\begin{align*}
\Lambda_{\po_{=}(J)}(f,\vec{g})&\lesssim |I_J| \,\overrightarrow{\operatorname{size}}_{\po_{=}(J)}(f)\,\widetilde{\operatorname{size}}_{\po_{=}(J)}(|\vec g|)\\
&\lesssim |I_J| \Big(\frac{1}{|I_J|}\int_{\R}|f(x)|^2\widetilde{\chi}_{I_J}(x)\,dx\Big)^{1/2}\Big(\frac{1}{|I_J|}\int_{\R}|\vec g(x)|\widetilde{\chi}_{I_J}(x)\,dx\Big).
\end{align*}
 Observe that, given $\vec g$ and any $J$ such that $\supp(|\vec g|)\cap 3J=\emptyset$, we have
\begin{align*}
\frac{1}{|J|}\int_{\R}|\vec g (x)|\widetilde{\chi}_{J}(x)\,dx&=\frac{1}{|J|}\int_{3J}|\vec g(x)|\widetilde{\chi}_{J}(x)\,dx\\
&\quad +\sum_{k\ge1}\frac{1}{|J|}\int_{2^{k+1}(3J)\setminus 2^k(3I)}|\vec g(x)|\Big(1+\frac{\dist(x,J)}{|J|}\Big)^{-100}\,dx\\
&\le \sum_{k\ge 1}\frac{1}{|J|}\int_{2^{k+1}(3J)}|\vec g (x)|\Big(1+\frac{3\cdot 2^k|J|}{|J|}\Big)^{-100}\,dx\\
&\le 3^{-99}\inf_{x\in 3J}M_1|\vec g|(x).
\end{align*}
Taking this into account we conclude that 
$$
\Lambda_{\po_{=}(J)}(f,\vec{g})\lesssim 3^{-98}|I_J|\inf_{x\in 3I}M_2f(x)\inf_{x\in 3I}M_1|\vec g|(x).
$$

Let us follow again the ideas in \cite{CDpO}. Let $\mathcal{J}=\{J:J=I_P \text{ for some } P\in \po_{\le}(I)\}$. We partition 
$$
\mathcal{J}_k=\{J\in \mathcal{J}:2^kJ\subset I, 2^{k+1}J\not\subset I\}, \qquad \po_{\le,k}(I)=\{P\in \po_{\le}(I):I_P\in \mathcal{J}_k\}.
$$
Note the following properties of the intervals $J\in \mathcal{J}_k$:
$$
\dist(J, \supp |\vec g| 1_{I^{\operatorname{out}}})\sim 2^k |J|, \quad J\in \mathcal{J}_k \text{ have finite overlap and } \sum_{J\in \mathcal{J}_k}|J|\lesssim |I|,
$$
and 
$$
\inf_{x\in 3J}M_2 f(x)\lesssim 2^k\inf_{x\in 3I}M_2 f(x), \quad \inf_{x\in 3J}M_1 |\vec g|(x)\lesssim 2^k\inf_{x\in 3I}M_1 |\vec g|(x).
$$
Then, by using \eqref{eq:53} and the above properties, we obtain
\begin{multline*}
\Lambda_{\po_{\le}(I)}^{\vec{t}}(f,\vec g)\lesssim \sum_{k\ge0}\sum_{J\in \mathcal{J}_k}\Lambda_{\po_{=}(J)}(f1_{I^{t_1}},|\vec g|1_{I^{\operatorname{out}}})\\\lesssim \sum_{k\ge0}\sum_{J\in \mathcal{J}_k}2^{-100k}|J|\inf_{x\in 3J}M_2f(x)\inf_{x\in 3J}M_1|\vec g|(x)\le |I|\inf_{x\in 3I}M_2f(x)\inf_{x\in 3I}M_1|\vec g|(x).
\end{multline*}
The proof is complete.
\end{proof}
With Lemma \ref{lem:claim3} we can conclude the proof of Lemma \ref{lem:claim2}. Indeed, 
$$
\Lambda_{\po_{\le}(I)}^{\vec{t}}(f,\vec g)\lesssim |I|\inf_{x\in 3I}M_2f(x)\inf_{x\in 3I}M_1|\vec g|(x) \lesssim |I|\langle f\rangle_{2,3Q}\langle |\vec g|\rangle_{3Q},
 $$
 where the last inequality follows from \eqref{eq:maxi}. Summing over $I$ yields the desired result in Lemma \ref{lem:claim2}, thus Lemma \ref{lem:recur}, therefore \eqref{eq:sparseJ} and finally Theorem \ref{thm:main2}.

\section{Quantitative weighted inequalities and further discussions}
\label{sec:Qweight}

\subsection{Quantitative weighted estimates}

One of the main consequences of sparse domination is that it provides weighted norm inequalities for operators under consideration with explicit constants involving the weighted $A_p$ characteristic. For a summary of this and other applications, we refer the reader to \cite[Section 4]{BC} and references therein, see also \cite{FN}. Quantitative weighted estimates may be deduced for operators dominated by bilinear sparse forms as a consequence of the following theorem. Let $\mathrm{D}$ be the space of test function on $\R^n$ with the property that it is dense in $L^p(w)$ for all $1\le p<\infty$ and all weights $w\in A_{\infty}$. 
\begin{theorem}[\cite{BFP}]
\label{lem:BFP}
Let $1\le p_0<q_0\le \infty$. Let $T$ be a (sub)linear operator, initially defined on $\mathrm{D}$, with the following property: There exists 
 $c>0$ such that for all $f,g\in \mathrm{D}$ there exists a sparse collection $\mathcal{S}$ with 
$$
|\langle Tf,g\rangle|\le c\sum_{I\in \mathcal{S}} |I|\langle f\rangle_{p_0,I}\langle  g\rangle_{q'_0,I}. 
$$
Then for any $p_0<p<q_0$ and every weight $w\in A_{p/p_0}$,
$$
\|T\|_{L^p(w)\to L^{p}(w)}\lesssim [w^{(q_0/p)'}]_{A_{\phi(p)}}^{\max\big(\frac{1}{p-p_0},\frac{q_0-1}{q_0-p}\big)/\big(\frac{q_0}{p}\big)'},
$$
where 
$$
\phi(p)=(q_0/p)'(p/p_0-1)+1,
$$ 
and the exponent in the last estimate is optimal for sparse operators. 
The constants involved in the inequalities depend on $p_0, q_0$ and $p$.
\end{theorem}
Actually, Theorem \ref{lem:BFP} above can be stated also in terms of the characteristic of weights belonging to the intersection of the $A_p$ and the reverse H\"older classes, but we prefer to simplify the presentation, stating our results only in terms of the $A_p$ constant.
It is also noteworthy to observe that Theorem \ref{lem:BFP} was improved to a mixed $A_p-A_{\infty}$ type estimate, see \cite[Theorem 1.2]{Li}.

The question whether the quantitative boundedness in Theorem \ref{lem:BFP} is sharp or not arises immediately. The connection in Theorem \ref{thm:con} below between the weighted strong type estimates for $T$ and the asymptotic behavior of the unweighted $L^p$ operator norm at the endpoints $p=p_0$ and $p=q_0$ is established in \cite[Theorem 5.2]{FN} which, in its turn, is a generalisation of the results in \cite{LPR}. 

\begin{definition}
Let $1\le p_0<q_0\le \infty$. Let $T$ be a bounded operator on $L^p$ for all $p_0<p<q_0$. We define
$$
\alpha_T(p_0):=\sup \{\alpha\ge 0|\mid, \forall \varepsilon >0, \limsup_{p\to p_0}(p-p_0)^{\alpha-\varepsilon}\|T\|_{L^p\to L^p}=\infty\}.
$$
For $q_0<\infty$ we define
\[
	\gamma_T(q_0) := \sup \{ \gamma \geq 0 \mid \forall \varepsilon >0,  \limsup_{p \to q_0} \,(q_0-p)^{\gamma-\varepsilon}\|T\|_{L^p\to L^p}  = \infty\},
\]
and for $q_0=\infty$
\[
	\gamma_T(\infty) := \sup \{ \gamma \geq 0 \mid \forall \varepsilon >0,  \limsup_{p \to \infty}  \frac{\|T\|_{L^p\to L^p} }{ p^{\gamma-\varepsilon}} = \infty\}.
\]
\end{definition}
\begin{theorem}[\cite{FN}]
\label{thm:con}
Let $T$ be a bounded operator on $L^p$ for all $p_0<p<q_0$. Suppose that for some $p_0<s<q_0$ and for all $w\in A_{s/p_0}$,
$$
\|T\|_{L^s(w)\to L^{s}(w)}\lesssim  [w^{(q_0/s)'}]_{A_{\phi(s)}}^{\beta/(q_0/s)'}.
$$
Then 
$$
\beta\ge \max \Big(\frac{p_0}{s-p_0} \alpha_T(p_0), \big(\frac{q_0}{s}\big)'\gamma_T(q_0)\Big).
$$
\end{theorem}
For the Rubio de Francia's Littlewood--Paley square function $T$, we know due to Bourgain~\cite{B}  that $\|T\|_{L^p\to L^p}\simeq p$ for $2\leq p <\infty$. 
Observe that, in view of this, from Theorem \ref{thm:main}, Theorem \ref{lem:BFP} and Theorem~\ref{thm:con}, we infer the following estimate for the Rubio de Francia's Littlewood--Paley square function $T$ (it is exactly Corollary \ref{cor:what0} but we state it here for the sake of readeness).
 \begin{corollary}
 \label{cor:what}
 For $2<p<\infty$ and any $w\in A_{p/2}$
 $$
 \|T\|_{L^p(w)\to L^{p}(w)}\lesssim [w]_{A_{p/2}}^{\max\big(\frac{1}{p-2},1\big)}.
 $$
 In particular, for $3\le p<\infty$, we have that $\max\big(\frac{1}{p-2},1\big)=1$ is the sharp exponent in this range.
 \end{corollary}

Further,  we could also try to apply sharp extrapolation theorem to check the sharpness of the quantitative estimate in Corollary \ref{cor:what} in the range $2<p<3$.
For instance, the following extrapolation result was proved by J. Duoandikoetxea in~\cite[Corollary 4.2]{D2}.

\begin{theorem}[\cite{D2}]
\label{shext}
Let $1\le \lambda<\infty$ and $\lambda\le p_0<\infty$. Assume that for some $f,g$ and for all weights $w\in A_{p_0/\lambda}$,
$$\|f\|_{L^{p_0}(w)}\le CN([w]_{A_{p_0/\lambda}})\|g\|_{L^{p_0}(w)},$$
where $N$ is an increasing function and the constant $C$ does not depend on $w$. Then for all $\lambda\le p<\infty$ and all $w\in A_{p/\lambda}$,
$$\|f\|_{L^p(w)}\le C_1N\Big(C_2[w]_{A_{p/\lambda}}^{\max\big(1,\frac{p_0-\lambda}{p-\lambda}\big)}\Big)\|g\|_{L^p(w)}.
$$
\end{theorem}
By Corollary \ref{cor:what}, we know that for $\lambda=2$ and $p_0=3$, we have, for $w\in A_{3/2}$,
$$\|Tf\|_{L^{3}(w)}\le C[w]_{A_{3/2}}\|f\|_{L^{3}(w)}.
$$ 
Then, Theorem \ref{shext} yields for all $2\le p<\infty$ and all $w\in A_{p/2}$,
$$\|Tf\|_{L^p(w)}\le C_1\big(C_2[w]_{A_{p/2}}^{\max\big(1,\frac{1}{p-2}\big)}\big)\|f\|_{L^p(w)}.
$$
Unfortunately, this approach does not give us anything better in this case. In particular, observe that in the case $p=2$ we are not obtaining the boundedness of the operator.

In the work of Frey and Nieraeth  \cite[Theorem 1.4]{FN} the following result concerning weighted weak type estimates was proved.

\begin{theorem}[\cite{FN}]
\label{thm:weakFN}
Let $1\le p_0<p<q_0\le \infty$ and $T$ be an operator with the same hypotheses as in Theorem \ref{lem:BFP}. Let $w\in A_1$. Then there is a constant $c>0$ so that
$$
\|T\|_{L^{p_0}(w)\to L^{p_0,\infty}(w)}\le c[w]_{A_1}^{\frac{1}{p_0}}[w]_{A_{\infty}}^{\frac{1}{p'_0}}\log(e+[w]_{A_{\infty}})^{\frac{2}{p_0}}.
$$
\end{theorem}
This yields the following weak-type end-point estimate for the Rubio's square function as a corollary. In particular, we recover the qualitative weighted weak type for $p=2$ in \cite[Theorem B (ii)]{K}.
 \begin{corollary}
 \label{cor:whatw}
 For $w\in A_{1}$
 $$
 \|T\|_{L^2(w)\to L^{2,\infty}(w)}\lesssim c[w]_{A_1}^{\frac{1}{2}}[w]_{A_{\infty}}^{\frac{1}{2}}\log(e+[w]_{A_{\infty}}).
 $$
 \end{corollary}
 We do not know whether the quantitative estimate above is sharp, or if the logarithm term can be removed. 

\subsection{Further discussions}
\label{sub:conje}

Recall that in \cite{Rubio}, Rubio de Francia proved the following.
\begin{theorem}[\cite{Rubio} Theorem 6.1]
\label{thm:Rubio}
Let $2<p<\infty$ and $w\in A_{p/2}(\R)$. Then the square function $T$ is bounded on $L^p(w)$.
\end{theorem}
In the same paper, he conjectured that the boundedness on $L^2(w)$ for $w\in A_1$ should also hold (see \cite[Section 6, p. 10]{Rubio} and \cite[Chapter 8, Section 8.2, p. 186--187] {Duobook}). To our best knowledge, this is still a conjecture.  He also pointed out that if we consider the particular case of congruent intervals, then the  $L^2(w)$ boundedness of the square function for $w\in A_1$ holds. Moreover, keeping a track of weighted characteristic constant in his proof (shown in his earlier paper \cite[Theorem A]{Rubio2}), one can get the linear growth in terms of $[w]_{A_1}$. Further, the operator norm is independent of the common length of intervals under consideration.  If we combine this observation with weighted boundedness of the lacunary square functions, we get more evidences supporting the conjecture. More precisely, let $w_k$ denote a lacunary sequence of intervals partitioning the real line.  For each $k\in \Z$, consider the decomposition of $w_k$ into congruent intervals $\{w_{k,n}\}_{n=1}^{N_k}$ forming a new sequence $\{w_{k,n}\}$. We allow the length of congruent intervals to vary with $k$. Then the associated square function is $L^2(w)$ bounded for $w\in A_1$. For, let $T_{k,n}(f) $ and $T_k(f)$ denote the Fourier multiplier operators with symbol ${\bf 1}_{w_{k,n}}$ and ${\bf 1}_{w_{k}}$ respectively. We may write $T_{k,n}$ as composition of $T_{k,n}$ with $T_k$. Then we use the weighted estimates for the sequence indexed by $n$ and then for the sequence indexed by $k$. Indeed, for the associated square function $T$ we have 
\begin{align*}
\|T(f)\|^2_{L^2(w)}&=\sum\limits_{k,n} \int_{\R}|T_{k,n}(f)|^2 w(x)dx \\
&=\sum\limits_{k,n} \int_{\R}|T_{k,n}(T_k f)|^2 w(x)dx \\
&\lesssim  [w]^2_{A_1} \sum\limits_{k} \int_{\R}|T_{k}(f)|^2 w(x)dx \\
&\lesssim  [w]_{_{A_1}} ^{5}\int_{\R}|f|^2 w(x)dx.
\end{align*}
Here, in the first inequality we have used the $L^2(w)$ boundedness of the square function with congruent intervals for $A_1$ weights. The second inequality follows from $L^2(w)$ boundedness of lacunary square function (see~Theorem~\ref{qw:lerner}) for $A_2$ weights along with the fact that $[w]_{A_2}\leq [w]_{A_1}$. These simple examples support Rubio's conjecture. 

As we have seen, the sparse domination method in the current form yields weak-type estimates at the end-points in the range under consideration. Observe that the quantitative weak-type result for $p=2$ in Corollary \ref{cor:whatw} could prevent us from conjecturing the linear dependence on the characteristic of the weight. However, we do not know whether such a quantitative weak type estimate is sharp, or if the logarithm term can be removed. Actually, this is another interesting open problem to be attained.  

On the other hand, the quantitative estimates given in Corollary \ref{cor:what0} for the interval $2<p<3$, blow-up as $p\rightarrow 2$. In view of the Rubio's conjecture this bound seems to be far from being sharp and in our opinion that it is due to the techniques employed in this paper. We think that we would require a new set of ideas in order to make any progress in this direction. 

In conclusion, we believe that Rubio's conjecture should hold, but at this moment we do not have an intuition of which should be the correct dependence on the characteristic of the $A_1$ weight. 

We finish the paper with a summary of the open problems:
\begin{enumerate}
\item Prove or disprove the following: For $w\in A_1$
$$
\|T\|_{L^2(w)\to L^2(w)}\lesssim \Phi([w]_{A_1}),
$$
and in case of a positive answer, determine the sharp dependence on $[w]_{A_1}$.
\item Determine whether the quantitative weak type estimate is sharp
$$
 \|T\|_{L^2(w)\to L^{2,\infty}(w)}\lesssim c[w]_{A_1}\log(e+[w]_{A_{1}}).
 $$
 \item Determine whether the quantitative strong type estimate is sharp in the range $2<p<3$
 $$
 \|T\|_{L^p(w)\to L^{p}(w)}\lesssim [w]_{A_{p/2}}^{\max\big(\frac{1}{p-2},1\big)}.
 $$
 Of course, the latter is closely related to the first one.
\end{enumerate}

\section*{Acknowledgements}

This work was initiated during the visit of the second author to R. Garg and S. Shrivastava at IISER Bhopal in July 2018. The second author is very grateful for the kind hospitality. The authors would also like to thank  Andrei Lerner for several discussions and suggestions related to this project and to Camil Muscalu for offering clarifications about the helicoidal method. We are also greatly indebted to the referee for helpful and valuable remarks.

The first author was supported in part by the INSPIRE Faculty Award from the Department of Science and Technology (DST), Government of India. The second author was supported by the Basque Government through BERC 2018--2021 program, by Spanish Ministry of Science, Innovation and Universities through BCAM Severo Ochoa accreditation SEV-2017-2018 and the project MTM2017-82160-C2-1-P funded by AEI/FEDER, UE, and by 2017 Leonardo grant for Researchers and Cultural Creators, BBVA Foundation. The Foundation accepts no responsibility for the opinions, statements and contents included in the project and/or the results thereof, which are entirely the responsibility of the authors. The third author was supported by Science and Engineering Research Board (SERB), Government of India, under the grant MATRICS: MTR/2017/000039/Math.

\end{document}